\documentclass[a4paper,12pt]{article}
\usepackage{cmap}                        
\usepackage[cp1251]{inputenc}            
\usepackage[english]{babel}
\usepackage[left=2.5cm,right=2.5cm,top=1.8cm,bottom=1.8cm]{geometry} 
\usepackage[ruled,vlined]{algorithm2e}
\usepackage{amssymb}
\usepackage{amsmath, amsthm}
\theoremstyle{plain}
\newtheorem{thm}{Theorem}
\newtheorem{lem}{Lemma}

\newtheorem{cor}{Corollary}

\theoremstyle{definition}

\newtheorem{remark}{Remark}
\newtheorem{defn}{Definition}
\newtheorem{pr}{Problem}

\begin{document}

\begin{center}\Large
\textbf{$\sigma$-properties of finite groups in  polynomial time\footnote{This work is supported by BFFR (project $\Phi$23PH$\Phi$-237).}}\normalsize

\smallskip
Viachaslau I. Murashka

 \{mvimath@yandex.ru\}

Faculty of Mathematics and Technologies of Programming,
Francisk Skorina Gomel State University, Sovetskaya 104, Gomel,
246028, Belarus\end{center}

\begin{abstract}
  Let $H, K$ be  subgroups of the permutation group $G$ of degree $n$ with $K\trianglelefteq G$ and $\sigma$ be a partition of the set of all different prime divisors of $|G/K|$. We prove that in  polynomial time (in $n$) one can check $G/K$ for $\sigma$-nilpotency and $\sigma$-solubility; $H/K$ for $\sigma$-subnormality and $\sigma$-$p$-permutability in $G/K$. Moreover one can find the least partition $\sigma$ of $\pi(G/K)$ for which $G/K$ is $\sigma$-nilpotent. Also one can find the least partition $\sigma$ of $\pi(G/K)$ for which $H/K$ is $\sigma$-$p$-permutable in $G/K$.
 \end{abstract}

\section*{Introduction}

All groups considered are \textbf{finite}.  One of the main tools in the theory of groups is the Sylow theory which connects the structure of a group with the prime divisors of its order. Many important results about the structure of a group are given with the help of Sylow subgroups. In the last decade the $\sigma$-method obtained a great development (for example, see \cite{ballester2020sigma,  BallesterBolinches2022, Ferrara2023, Guo2020,  math8122165, murashka2018generalization,  skiba2015sigma} and other). Its main idea is to study the structure of a group according to some partition $\sigma$ of the set of its prime divisors. So the Sylow theory and its applications are the particular case of this method when each element of $\sigma$ consists of one prime number. Another interesting case of this method is the Chunikhin's $\pi$-method \cite{CHUNIKHIN1969} which is the study of the structure of a group according to some set of primes $\pi$ (and its complement $\pi'$). In this case $\sigma$ consists of two elements.
The main advantage of $\sigma$-method is that we can chose $\sigma$ according to our task. Lets give the formal definition of $\sigma$-property

\begin{defn}
Let $\sigma$ be a partition of $\pi(G)$.
  By $\sigma$-property we will understand the property $\theta=\theta(\sigma)$ of a subgroup $H$ in a group $G$ such that
  \begin{enumerate}
    \item If $H$ has properties $\theta(\sigma^1)$ and $\theta(\sigma^2)$ in a group $G$, then it has  property $\theta(\sigma^1\cap\sigma^2)$ in a group $G$.

    \item If $|\sigma|=1$, then $H$ has property $\theta(\sigma)$  in a group $G$.
  \end{enumerate}
\end{defn}

According to this definition every subgroup $H$ of a group $G$ has the given  $\sigma$-property for the right choice of $\sigma$. Moreover if the property $\theta$ is fixed there exists the least  partition $\sigma$ for which $H$ has the property $\theta(\sigma)$ in $G$.

\begin{pr}\label{pr1}
Let $\theta$ be some $\sigma$-property and $H$ be a subgroup of a group $G$.

\begin{enumerate}
  \item Is there an effective algorithm which checks if $H$ has the property $\theta(\sigma)$ in $G$ when $\sigma$ is given?

  \item Is there an effective algorithm which finds the least $\sigma$ for which $H$ has the property $\theta(\sigma)$ in $G$?
\end{enumerate}
\end{pr}

The aim of this paper is to give the answers to this question when $\theta\in\{\sigma$-nilpotency, $\sigma$-solubility, $\sigma$-subnormality, $\sigma$-$p$-permutability$\}$ and by effective algorithm we mean an algorithm which works in  polynomial time (in $n$) for a permutation group of degree $n$.

\section{The Main Results}

Let $\sigma$ be a partition of the set of all primes.
Recall \cite[Proposition 2.3]{skiba2015sigma} that a group $G$ is called {\bf$\sigma$-\emph{nilpotent}} if it has a normal Hall $\sigma_i$-subgroup for every $\sigma_i\in\sigma$;
 a group $G$ is called {\bf $\sigma$-\emph{soluble}}  \cite[Definition 2.3]{skiba2015sigma} if every its chief factor is a $\sigma_i$-group for some $\sigma_i\in\sigma$.
A subgroup $H$ of a group $G$ is said to be {\bf$\sigma$-\emph{subnormal}} \cite[Definition 1.1]{skiba2015sigma} if there is a chain of subgroups $H=H_0\leq\dots\leq H_n=G$ with $H_{i-1}\trianglelefteq H_i$ or $H_i/(H_{i-1})_{H_i}$ is a $\sigma_i$-group for some $\sigma_i\in\sigma$ for all $1\leq i\leq n$.

Recall that a subgroup $H$ of a
group $G$ is called {\bf$\sigma$-\emph{permutable}} \cite{skiba2015sigma} if $G$ has a Hall $\sigma_i$-subgroup $H_i$ with $HH_i^x=H_i^xH$
  for every $x\in G$ and $\sigma_i\in\sigma$. Hence the concept of $\sigma$-permutability is defined not for all groups in the general. The
  exception is the case when  every element of $\sigma$ consists of one element. In this case the concepts of $\sigma$-permutable and $S$-permutable subgroups coincides.

  One of the main properties of Hall $\pi$-subgroups is that their images  are maximal among $\pi$-subgroups in every epimorphic image of a group. A subgroup $H$ of $G$ is said to be a $\mathfrak{G}_\pi$-\emph{projector} of $G$ \cite[III, Definition 3.2]{Doerk1992} if $HN/N$ is  maximal among $\pi$-subgroups of  $G/N$ for every $N\trianglelefteq G$. Note that   $\mathfrak{G}_\pi$-projectors exist in every group by  \cite[III, Theorem 3.10]{Doerk1992}.
A subgroup $H$ of a group $G$ is said to be {\bf$\sigma$-$p$-\emph{permutable}} \cite[Definition 2]{murashka2018generalization}
if for all $\sigma_i\in\sigma$  there is a $\mathfrak{G}_{\sigma_i}$-projector $H_i$ of $G$ such that $HH_i^x=H_i^xH$
  for every $x\in G$. According to \cite[Theorem 1]{murashka2018generalization}   $\sigma$-$p$-permutability does not depend on the choice of projectors, i.e.
  a subgroup $H$ of a group $G$ is   $\sigma$-$p$-permutable
if and only if $HH_i^x=H_i^xH$ for every $\mathfrak{G}_{\sigma_i}$-projector $H_i$ of $G$  for all  $\sigma_i\in\sigma$ and $x\in G$. Hence if  $G$ has a Hall $\sigma_i$-subgroup for all $\sigma_i\in\sigma$, then the sets of $\sigma$-permutable and $\sigma$-$p$-permutable subgroups coincides.

\begin{thm}\label{thm1}
  The $\sigma$-nilpotency, $\sigma$-solubility, $\sigma$-subnormality and $\sigma$-$p$-permutability are $\sigma$-properties.
\end{thm}

If a group $G$ is fixed, the above mentioned $\sigma$-properties can be defined only using the partition $\sigma(G)=\{\sigma_i\cap\pi(G)\mid\sigma_i\in\sigma\}$ of $\pi(G)$. Therefore it makes sense in the solution of Problem \ref{pr1} only consider the partitions of $\pi(G)$.  The answer to Problem \ref{pr1}(1) is given in

\begin{thm}\label{thm2}
  Let $H, K\leq G\leq S_n$ with $K\trianglelefteq G$ and $\sigma$ be a partition of $\pi(G)$. In polynomial time in $n$ one can check:
  \begin{enumerate}
  \item If $G/K$ is $\sigma$-nilpotent.

  \item If $H/K$ is $\sigma$-subnormal in $G/K$.

  \item If $H/K$ is $\sigma$-$p$-permutable in $G/K$.

 \item If $G/K$ is $\sigma$-soluble. In case of affirmative answer if $H/K$ is $\sigma$-permutable in $G/K$.
  \end{enumerate}
\end{thm}

The answer to Problem \ref{pr1}(2) is given in

\begin{thm}\label{thm3}
  Let $H, K\leq G\leq S_n$ with $K\trianglelefteq G$. In every of the following cases   in polynomial time in $n$
  \begin{enumerate}
  \item  One can find  the least partition  $\sigma$  of $\pi(G)$   for which  $G/K$ is $\sigma$-nilpotent.

  \item One can find  the least partition  $\sigma$  of $\pi(G)$ for which  $G/K$ is $\sigma$-soluble.

  \item One can find  the least partition  $\sigma$  of $\pi(G)$  for which   $H/K$ is
   $\sigma$-$p$-permutable in $G/K$. In particular, if $G/K$ is soluble, then one can find  the least partition  $\sigma$  of $\pi(G)$  for which  $H/K$ is $\sigma$-permutable in $G/K$.

  \end{enumerate}
\end{thm}

\section{Preliminaries}

\subsection{Group Theory}

Recall that $S_n$ is the symmetric group of degree $n$; $\langle X\rangle$ is a group generated by $X$; $H^G$ denotes the smallest normal subgroup of $G$ which contains $H$;
$H_G$ denotes the greatest normal subgroup of $G$  contained in $H$; $\mathbb{P}$ is the set of all primes;  $\mathrm{O}_\pi(G)$ is the greatest normal $\pi$-subgroup of $G$ and $\mathrm{O}^\pi(G)$ is the smallest normal subgroup of $G$  of $\pi$-index for $\pi\subseteq\mathbb{P}$; $\pi(n)$ is the set of all different  prime divisors of a natural number $n$; $\pi(G)=\pi(|G|)$ for a group $G$. If $\sigma$ is a partition of $\pi(G)$ and $H/K$ is a section  of a group $G$, then $\sigma(H/K)=\{\sigma_i\cap\pi(H/K)\mid\sigma_i\in\sigma\}$ is a partition of $\pi(H/K)$.

\subsection{Algorithms}

We use standard computational conventions of abstract finite groups equipped
with poly\-nomial-time procedures to compute products and inverses of elements (see \cite[Chapter 2]{Seress2003}).
For both input and output, groups are specified by generators. We will consider only $G=\langle S\rangle\leq S_n$ with $|S|\leq n^2$. If necessary, Sims' algorithm \cite[Parts 4.1 and 4.2]{Seress2003} can be used to arrange that $|S|\leq n^2$. Quotient groups are specified by generators of a group and its normal subgroup.
We need the following well known basic tools in our proofs (see, for example \cite{Kantor1990a} or \cite{Seress2003}).

\begin{thm}\label{Basic}
  Given $G = \langle S\rangle\leq S_n$, in polynomial time one can solve the following problems:

  \begin{enumerate}
    \item  Find $|G|$.

    \item Given normal subgroups $A$ and $B$ of $G$ with $A\leq B$,   find a composition series for $G$ containing them.

    \item Given $T\subseteq G$ find $\langle T\rangle^G$.

    \item (mod CFSG) Given $N, K \leq S_n$ such that $N/K$ is normalized by $G/K$,
     find $C_{G/K}(N/K)$ \cite[P6(i)]{Kantor1990a}.

    \item (mod CFSG) Given $H\leq G$ find $H_G$ \cite[P5(i)]{Kantor1990a}.

    \item (mod CFSG) Given a prime $p$ dividing $|G|$, find a Sylow $p$-subgroup $P$ of $G$ \cite{Kantor1990}.

\item Given   $H=\langle S_1\rangle, K=\langle S_2\rangle \leq G$ find $\langle H, K\rangle=\langle S_1, S_2\rangle$ and $[H, K]=\langle \{[s_1, s_2]\mid s_1\in S_1, s_2\in S_2\rangle^{\langle H, K\rangle}$.

\item Given $H, K\leq G$ with $K\trianglelefteq G$ find $H\cap K$.
  \end{enumerate}
\end{thm}

Note that $H\subseteq K$ iff $\langle H, K\rangle=K$. From 1 and 7 of Theorem \ref{Basic} directly follows

\begin{cor}\label{CorBasic}
  Given $G, G_1, G_2\leq S_n$, in polynomial time one can solve the following problems:
 \begin{enumerate}
   \item Check if $G_1=G_2$;
   \item Check if $G_1\subseteq G_2$;
   \item Compute $\pi(G)$ and $\pi(G/K)$ for $K\trianglelefteq G$.
 \end{enumerate}
\end{cor}

 \begin{remark}\label{rem1}
   Note that  $\mathrm{O}^\pi(G)=\langle P_i\mid P_i$ is a Sylow $p_i$-subgroup for $p_i\in\pi(G)\setminus\pi\rangle^G$. Hence
 it can be computed in polynomial time by 3 and 6 of Theorem \ref{Basic}.
 \end{remark}

According to  the following result \cite{Babai1986}, the lengthes of all chains of subgroups  in a permutation group are bounded:

\begin{lem}[\cite{Babai1986}]\label{chain}
   Given $G \leq S_n$ every chain of subgroups of $G$ has at most $2n-3$  members for $n\geq 2$.
\end{lem}

\begin{lem}\label{decomp}
  Let $G=\langle S\rangle\leq S_n$ and $\sigma=\{\sigma_1,\dots, \sigma_k\}$ be a partition of $\pi(G)$.
   In polynomial time one can find sets $S_{\sigma_1},\dots, S_{\sigma_k}$ such that   $S_{\sigma_i}$ consists of $\sigma_i$-elements for all $\sigma_i\in\sigma$ and $G=\langle S_{\sigma_1},\dots, S_{\sigma_k}\rangle$.
\end{lem}

\begin{proof}
Since $|S|\leq n^2$, it is enough to prove that  for $s\in S$   in   polynomial time one can find elements $s_{\sigma_1}\dots s_{\sigma_k}$ such that $\langle s\rangle=\langle s_{\sigma_1},\dots, s_{\sigma_k}\rangle$ where $s_{\sigma_i}$ is an $\sigma_i$-element of $G$ (possibly $s_{\sigma_i}=1$) for every $\sigma_i\in\sigma$.

  By 1 of Theorem \ref{Basic} in  polynomial time one can find $m=|\langle s \rangle|$. Since every prime divisor of $|\langle s \rangle|$ is not greater than $n$, in  polynomial time one can decompose $m=\prod p_i^{\alpha_i}$. From $\langle s \rangle\leq S_n$ it follows that $\alpha_i\leq n$ for all $i$. Now let $s_{p_i}=s^{m/p_i^{\alpha_i}}$. Note that $s_{p_i}$ generates the Sylow $p_i$-subgroup of $\langle s\rangle$. Now let $s_{\sigma_i}=\prod_{p_j\in\sigma_i} s_{p_j}$.
\end{proof}

\begin{remark}
  Note that the constructed generating set $\{S_{\sigma_1},\dots, S_{\sigma_k}\}$ of $G$ has at most $n^3$ elements.
\end{remark}

\subsection{$\sigma$-properties}

Let $\sigma^1$ and $\sigma^2$ be partitions of the same set. Then $\sigma^1\cap\sigma^2=\{\sigma_i^1\cap\sigma_j^2\mid\sigma_i^1\in\sigma^1, \sigma_j^2\in\sigma^2\}$ is  also a partition of this set.
We say that $\sigma^1\leq \sigma^2$ if for every $\sigma_i^1\in\sigma^1$ there is $\sigma_j^2\in\sigma^2$ with $\sigma_i^1\subseteq\sigma^2_j$.

\begin{thm}\label{thm4}
Let $\sigma=\{\sigma_1,\dots, \sigma_k\}$ be a partition of $\pi(G)$.  Assume that $G=\langle S_{\sigma_1},\dots, S_{\sigma_k}\rangle$ where every element of $S_{\sigma_i}$ is a $\sigma_i$-element for every $\sigma_i\in\sigma$.  Then $G$ is a $\sigma$-nilpotent group iff   $\langle S_{\sigma_i}\rangle$ is a $\sigma_i$-group and $[S_{\sigma_i}, S_{\sigma_j}]=1$ for every $\sigma_i,\sigma_j\in\sigma$ with $\sigma_i\neq\sigma_j$.
\end{thm}

\begin{proof}
  Assume that   $\langle S_{\sigma_i}\rangle$ is a $\sigma_i$-group and $[S_{\sigma_i}, S_{\sigma_j}]=1$ for every $\sigma_i,\sigma_j\in\sigma$ with $\sigma_i\neq\sigma_j$. Then $\langle S_{\sigma_i}\rangle$ is a     normal  $\sigma_i$-subgroup of $G$. Note that $G=\prod_{\sigma_i\in\sigma}\langle S_{\sigma_i}\rangle$. Hence $\langle S_{\sigma_i}\rangle$ is a normal Hall $\sigma_i$-subgroup of $G$. Thus $G$ is $\sigma$-nilpotent.

  Assume that $G$ is $\sigma$-nilpotent. Then all $\sigma_i$-elements of $G$ generates a normal Hall $\sigma_i$-subgroup of $G$ and every $\sigma_i$-element of $G$ commute with every  $\sigma_j$-element of $G$ for $i\neq j$. Thus  $\langle S_{\sigma_i}\rangle$ is a $\sigma_i$-group and $[S_{\sigma_i}, S_{\sigma_j}]=1$ for every $\sigma_i,\sigma_j\in\sigma$ with $i\neq j$.
\end{proof}

\begin{lem}[{\cite[Lemma 2.6]{skiba2015sigma}}]\label{lem1}
  Let $H$ be a $\sigma$-subnormal subgroup of a group $G$ and $N\trianglelefteq G$.

  \begin{enumerate}
 \item If $H\leq K\leq G$, then $H$ is $\sigma$-subnormal in  $K$.

\item  $HN/N$ is  $\sigma$-subnormal in  $G/N$

 \item If $A/N$ is  $\sigma$-subnormal in $G/N$, then $A$ is a $\sigma$-subnormal in $G$.

  \end{enumerate}
\end{lem}

The following 4 lemmas contain an important information about   $\sigma$-$p$-permutability.

\begin{lem}[{\cite[Lemma 1]{murashka2018generalization}}]\label{l1} Let  $N$ be a normal subgroup of a group $G$.
\begin{enumerate}
\item   If $H$ is a $\sigma$-$p$-permutable subgroup of   $G$, then $HN/N$ is a $\sigma$-$p$-permutable subgroup of   $G/N$.
\item If   $H/N$ is a $\sigma$-$p$-permutable subgroup of   $G/N$, then $H$ is a $\sigma$-$p$-permutable subgroup of   $G$.
\end{enumerate}
\end{lem}

\begin{lem}[{\cite[Theorem 2]{murashka2018generalization}}]\label{t1} Let $H$ be a $\sigma$-$p$-permutable subgroup of a group   $G$. Then $H^G/H_G$ is $\sigma$-nilpotent.
\end{lem}

\begin{lem}[{\cite[Theorem 3(2)]{murashka2018generalization}}]\label{prop1} Let  $H$ be a $\sigma$-nilpotent subgroup of a group   $G$. Then $H$ is   $\sigma$-$p$-permutable in $G$    if and only if every Hall $\pi_i$-subgroup  of $H$ is $\sigma$-$p$-permutable in $G$  for all $\pi_i\in\sigma$.
\end{lem}

\begin{lem}[{\cite[Lemma 5]{murashka2018generalization}}]\label{l2} Let $H$ be a $\sigma$-$p$-permutable subgroup of a group $G$. Then $\mathrm{O}^{\pi_i}(G)\leq N_G(\mathrm{O}_{\pi_i}(H))$ for every $\pi_i\in\sigma$.        \end{lem}

Now we are ready to present a criterion for $\sigma$-$p$-permutability.

\begin{thm}\label{thm5}
  A subgroup $H$ of a group $G$ is $\sigma$-$p$-permutable in $G$ if and only if $H^G/H_G$ is $\sigma$-nilpotent
  and $O(\sigma_i)^{\mathrm{O}^{\sigma_i}(G)}=O(\sigma_i)$ for every $\sigma_i\in\sigma(H/H_G)$ where   $O(\sigma_i)/H_G=\mathrm{O}_{\sigma_i}(H/H_G)$.
\end{thm}

\begin{proof}
  Assume that $H$ is $\sigma$-$p$-permutable in $G$. Then $H^G/H_G$ is a normal $\sigma$-nilpotent subgroup of $G/H_G$ by Lemma \ref{t1}.  Note that $H/H_G$ is $\sigma$-$p$-permutable in $G/H_G$ by Lemma \ref{l1}. Hence $$\mathrm{O}^{\sigma_i}(G/H_G)\leq N_{G/H_G}(O(\sigma_i)/H_G)$$ by Lemma \ref{l2}. Note that  $\mathrm{O}^{\sigma_i}(G)H_G/H_G=\mathrm{O}^{\sigma_i}(G/H_G)$. Hence $\mathrm{O}^{\sigma_i}(G)\leq N_{G}(O(\sigma_i))$. Thus $O(\sigma_i)^{\mathrm{O}^{\sigma_i}(G)}=O(\sigma_i)$ for every $\sigma_i\in\sigma(H/H_G)$.


  Assume that $H^G/H_G$ is $\sigma$-nilpotent and $O(\sigma_i)^{\mathrm{O}^{\sigma_i}(G)}=O(\sigma_i)$ for every $\sigma_i\in\sigma(H/H_G)$ where   $O(\sigma_i)/H_G=\mathrm{O}_{\sigma_i}(H/H_G)$. Note that $H/H_G$ is $\sigma$-nilpotent. Hence it is a direct product of its Hall subgroups $O(\sigma_i)/H_G=\mathrm{O}_{\sigma_i}(H/H_G)$ for all $\sigma_i\in\sigma(H/H_G)$. Therefore $O(\sigma_i)/H_G\leq \mathrm{O}_{\sigma_i}(H^G/H_G)\leq \mathrm{O}_{\sigma_i}(G/H_G)$ for all $\sigma_i\in\sigma(H/H_G)$. Thus $O(\sigma_i)/H_G$  lies in any $\sigma_i$-projector of $G/H_G$, in particular permutes with it, for all $\sigma_i\in\sigma(H/H_G)$.

  From  $O(\sigma_i)^{\mathrm{O}^{\sigma_i}(G)}=O(\sigma_i)$ and $\mathrm{O}^{\sigma_i}(G)H_G/H_G=\mathrm{O}^{\sigma_i}(G/H_G)$ it follows that $O(\sigma_i)/H_G$ is normalized by any $\sigma_i'$-subgroup of $G/H_G$. So it permutes with any $\sigma_j$-projector of $G/H_G$ for any $\sigma_j\in \sigma(G/H_G)$ with $\sigma_i\neq\sigma_j$. Thus $O(\sigma_i)/H_G$ is $\sigma$-$p$-permutable in $G/H_G$. Since $H/H_G$ is $\sigma$-nilpotent, it is $\sigma$-$p$-permutable in $G/H_G$ by Lemma \ref{prop1}. Thus $H$ is $\sigma$-$p$-permutable in $G$ by Lemma \ref{l1}.\end{proof}

\begin{lem}\label{coicide}
 Let $G$ be a group and $\sigma$ be a partition   $\pi(G)$. If $G$ has a Hall $\sigma_i$-subgroup for every $\sigma_i\in\sigma$, then the sets of $\sigma$-permutable and $\sigma$-$p$-permutable subgroups coincide.
\end{lem}

\begin{proof}
  Since a Hall $\pi$-subgroup is a $\mathfrak{G}_{\pi}$-projector, every $\sigma$-permutable subgroup is a $\sigma$-$p$-permutable subgroup. From \cite[Theorem 1]{murashka2018generalization} it follows that every $\sigma$-$p$-permutable subgroup permutes with every $\mathfrak{G}_{\sigma_i}$-projector for every $\sigma_i\in\sigma$. Hence it is a $\sigma$-permutable subgroup.
\end{proof}
\section{Proof of Theorem \ref{thm1}}

  1. Note that a group $G$ is $\sigma$-nilpotent for $\sigma=\{\pi(G)\}$. Let prove that if a group $G$ is $\sigma^1$-nilpotent and $\sigma^2$-nilpotent, then $G$ is $\sigma^1\cap\sigma^2$-nilpotent. Since $G$ is $\sigma^1$-nilpotent and $\sigma^2$-nilpotent, it has normal Hall $\sigma_i^1$-subgroups and
 $\sigma_j^2$-subgroups for every $\sigma_i^1\in \sigma^1$ and $\sigma_j^2\in \sigma^2$. Since the intersection of normal Hall subgroups is again a normal Hall subgroup, we see that $G$ has a normal  Hall $(\sigma_i^1\cap\sigma^2_j)$-subgroup for every
 $\sigma_i^1\in \sigma^1$ and $\sigma_j^2\in \sigma^2$. Thus $G$ is $\sigma^1\cap\sigma^2$-nilpotent.

 2. Note that a group $G$ is $\sigma$-soluble for $\sigma=\{\pi(G)\}$. Let prove that if a group $G$ is $\sigma^1$-soluble and $\sigma^2$-soluble, then $G$ is $\sigma^1\cap\sigma^2$-soluble. Since  $G$ is $\sigma^1$-soluble and $\sigma^2$-soluble, then every its chief factor is a $\sigma_i^1$-group and
a $\sigma_j^2$-group for some $\sigma_i^1\in \sigma^1$ and $\sigma_j^2\in \sigma^2$. Hence every chief factor of $G$ is a $(\sigma_i^1\cap\sigma^2_j)$-group for some
 $\sigma_i^1\in \sigma^1$ and $\sigma_j^2\in \sigma^2$. Thus $G$ is $\sigma^1\cap\sigma^2$-soluble.

3. From the definition of $\sigma$-subnormality it follows that every subgroup of a group $G$ is $\sigma$-subnormal for $\sigma=\{\pi(G)\}$. Let prove that if $H$ is $\sigma^1$-subnormal and $\sigma^2$-subnormal in $G$, then $H$ is $\sigma^1\cap\sigma^2$-subnormal in $G$. Assume that there exist groups that have  a $\sigma^1$-subnormal and $\sigma^2$-subnormal but not $\sigma^1\cap\sigma^2$-subnormal subgroup. Let   $G$ be the least order group among them. Hence it has   a $\sigma^1$-subnormal and $\sigma^2$-subnormal but not $\sigma^1\cap\sigma^2$-subnormal subgroup $H$.

 Assume that there is a subgroup $K$ with $H\leq K<G$ and $K\trianglelefteq G$. Then $H$ is $\sigma^1$-subnormal and $\sigma^2$-subnormal in $K$ by Lemma \ref{lem1}(1). So by our assumption $H$ is $\sigma^1\cap\sigma^2$-subnormal in $K$. Thus $H$ is $\sigma^1\cap\sigma^2$-subnormal in $G$ by definition, a contradiction.

 Therefore for every $K$ with $H\leq K<G$ we have $K\not\trianglelefteq G$. It means that a maximal chains for $\sigma^1$-subnormality and $\sigma^2$-subnormality of $H$ must contain a maximal subgroups $M_1$ and $M_2$ of $G$.
 Let $M$ be a maximal subgroup of $G$ with $H\leq M$. Suppose  that $M_G\neq 1$. Then $HM_G/M_G$ is  $\sigma^1$-subnormal and $\sigma^2$-subnormal in $G/M_G$ by Lemma  \ref{lem1}(2). By our assumption $HM_G/M_G$ is  $\sigma^1\cap \sigma^2$-subnormal in $G/M_G$. So $HM_G$ is  $\sigma^1\cap \sigma^2$-subnormal in $G$ by Lemma \ref{lem1}(3). Note that $H$ is $\sigma^1$-subnormal and $\sigma^2$-subnormal in $HM_G<G$ by Lemma \ref{lem1}(1). Hence $H$  is $\sigma^1\cap \sigma^2$-subnormal in $HM_G$. Thus $H$ is $\sigma^1\cap\sigma^2$-subnormal in $G$ by definition, a contradiction. It means that $M_G=1$ for every maximal subgroup of $G$ with $H\leq M$. Now $G\simeq G/(M_1)_G$ is a $\sigma^1_i$-group and $G\simeq G/(M_2)_G$ is a $\sigma^2_j$-group for some $\sigma_i^1\in \sigma^1$ and $\sigma_j^2\in \sigma^2$. So $G$ is a $\sigma^1_i\cap\sigma^2_j$-group. Therefore every subgroup of $G$ is  $\sigma^1\cap\sigma^2$-subnormal by the definition of $\sigma$-subnormality, the final contradiction. It means that if $H$ is $\sigma^1$-subnormal and $\sigma^2$-subnormal in $G$, then $H$ is $\sigma^1\cap\sigma^2$-subnormal in $G$.

4. From the definition of $\sigma$-$p$-permutability  it follows that every subgroup of a group $G$ is $\sigma$-$p$-permutable  for $\sigma=\{\pi(G)\}$. Assume that $H$ is $\sigma^1$-$p$-permutable and $\sigma^2$-$p$-permutable in $G$. Then $H^G/H_G$ is a $\sigma^1$-nilpotent and $\sigma^2$-nilpotent subgroup by Lemma \ref{t1}. Therefore $H^G/H_G$ is $\sigma^1\cap\sigma^2$-nilpotent by\,1. Hence $H/H_G$ is $\sigma^1\cap\sigma^2$-nilpotent. Note that $H/H_G$ is $\sigma^1$-$p$-permutable and $\sigma^2$-$p$-permutable in $G/H_G$ by Lemma \ref{l1}(1).

Let $H_{ij}/H_G=\mathrm{O}_{\sigma_i^1\cap\sigma^2_j}(H/H_G)$. Then $H_{ij}/H_G$ is a normal Hall subgroup of $H/H_G$.
Note that $\mathrm{O}^{\sigma_i^1}(G/H_G)\leq N_{G/H_G}(\mathrm{O}_{\sigma_i^1}(H/H_G))$ by Lemma \ref{l2}.  From $H_{ij}/H_G\textrm{ char } \mathrm{O}_{\sigma_i^1}(H/H_G)$ it follows that $\mathrm{O}^{\sigma_i^1}(G/H_G)\leq N_{G/H_G}(H_{ij}/H_G)$. By analogy $\mathrm{O}^{\sigma_j^1}(G/H_G)\leq N_{G/H_G}(H_{ij}/H_G)$. So $$\mathrm{O}^{\sigma_i^1\cap \sigma_j^2}(G/H_G)=\mathrm{O}^{\sigma_i^1}(G/H_G)\mathrm{O}^{\sigma_j^2}(G/H_G)\leq N_{G/H_G}(H_{ij}/H_G).$$ 
From $\mathrm{O}^{\sigma_i^1\cap \sigma_j^2}(G/H_G)=\mathrm{O}^{\sigma_i^1\cap \sigma_j^2}(G)H_G/H_G$ it follows that $H_{ij}^{\mathrm{O}^{\sigma_i^1\cap \sigma_j^2}(G)}=H_{ij}$. 
Thus $H$ is  $\sigma^1\cap\sigma^2$-$p$-permutable by Theorem \ref{thm5}. 

\section{Proof of Theorem \ref{thm2}}

1. According to Lemma \ref{decomp} given a generating set $S$ (of polynomial in $n$ size) of $G$ one can find in polynomial time the generating set $S'=\cup_{\sigma_i\in\sigma}S_{\sigma_i}$  (of polynomial in $n$ size) such that every its element is a $\sigma_i$-element for some $\sigma_i\in\sigma$.
Now according to Theorem \ref{thm4} we need only to check that $\langle S_{\sigma_i}\rangle K/K\simeq \langle S_{\sigma_i}\rangle/(\langle S_{\sigma_i}\rangle\cap K)$ is a $\sigma_i$-group and $ [S_{\sigma_i}, S_{\sigma_j}]\subseteq  K$ for every $\sigma_i, \sigma_j\in\sigma$ with $\sigma_i\neq\sigma_j$. All these can be done in  polynomial time by Theorem \ref{Basic}(7, 8) and Corollary \ref{CorBasic}.

\begin{algorithm}[H]
\caption{IsSigmaNilpotent($G, K, \sigma$)}
\SetAlgoLined
\KwResult{True, if $G/K$ is $\sigma$-nilpotent and false otherwise.}
\KwData{$K\trianglelefteq G=\langle S\rangle$, $\sigma=\{\sigma_1,\dots\sigma_k\}$ is a partition of $\pi(G)$}

Compute $S_{\sigma_1},\dots, S_{\sigma_k}$\;

\For{$i \in \{1,\dots, k\}$}
    {\If{$\pi(\langle S_{\sigma_i}\rangle/(\langle S_{\sigma_i}\rangle\cap K))\not\subseteq\sigma_i$}
          {\Return{False}\;}
       \For{$j \in \{i+1,\dots, k\}$ \textbf{and} $j\leq k$}
           {\If{$[S_{\sigma_i}, S_{\sigma_j}]\not\subseteq  K$}{\Return{False}\;}
       }
     }

 \Return{True};\
 \end{algorithm}

2. By 3 of Lemma \ref{lem1} it is enough to check that $H$ is $\sigma$-subnormal in $G$.
 According to the definition of $\sigma$-subnormality, if $H\neq G$ is $\sigma$-subnormal in $G$, then there exists a proper subgroup $M$ of $G$ with $H\leq M$ such that either $M\trianglelefteq G$ or $\mathrm{O}^{\sigma_i}(G)\leq M$ for some $\sigma_i\in\sigma$.
  Since $H^G$ and  $H\mathrm{O}^{\sigma_i}(G)$ for every $\sigma_i\in\sigma$ can be computed in  polynomial time by Theorem \ref{Basic}, the existence of such subgroup $M$ can be checked in  polynomial time. And in case of affirmative answer such $M$ will be present.
  Note that from 1 of Lemma \ref{lem1} if $H$ is $\sigma$-subnormal in $G$, the   $H$ is $\sigma$-subnormal in $M$. Now we can do a recursion. Note that every chain of subgroups of $G$ has a length at most $2n-3$ (for $n\geq 2$) by Lemma \ref{chain}.

\begin{algorithm}[H]
\caption{IsSigmaSubnormal($G, H, K, \sigma$)}
\SetAlgoLined
\KwResult{True, if $H/K$ is $\sigma$-subnormal in $G/K$ and false otherwise.}
\KwData{$H$ is a  subgroup of a group $G$, $K\trianglelefteq G$, $\sigma$ is a partition of $\pi(G)$}

\If{$H=G$}{\Return{True}\;}

\If{$H^G\neq G$}{\Return{IsSigmaSubnormal($H^G, H, K, \sigma$)\;}}

\For{$\sigma_i\in\sigma(|G:H|)$}{
   \If{$H\mathrm{O}^{\sigma_i}(G)\neq G$}{\Return{IsSigmaSubnormal($H\mathrm{O}^{\sigma_i}(G), K, H, \sigma$)\;}}
}

 \Return{False\;}
 \end{algorithm}

3. By    Lemma \ref{l1} it is enough to check that $H$ is $\sigma$-$p$-permutable  in $G$. The check for $\sigma$-$p$-permutability is described in Theorem \ref{thm5}.
From Theorem \ref{Basic} and 1 we can check $H^G/H_G$ for $\sigma$-nilpotency in polynomial time.
Given a generating set $S$ (of polynomial in $n$ size) of $H$ one can find in polynomial time the generating set $S'=\cup_{\sigma_i\in\sigma}S_{\sigma_i}$  (of polynomial in $n$ size) such that every its element is a $\sigma_i$-element for some $\sigma_i\in\sigma$. Since $H^G/H_G$ is $\sigma$-nilpotent and hence $H/H_G$ is $\sigma$-nilpotent, $\langle S_{\sigma_i}, H_G\rangle$ is the full inverse image of a Hall $\sigma_i$-subgroup $H_{\sigma_i}/H_G$ of $H/H_G$. Note that $H_{\sigma_i}/H_G=\mathrm{O}_{\sigma_i}(H/H_G)$. From Remark \ref{rem1} it follows that $\{ \mathrm{O}^{\sigma_i}(G)\mid \sigma_i\in\sigma(H/H_G)\}$ can be computed in polynomial time.
 Now by Theorem \ref{Basic} and Corollary \ref{CorBasic} we can check in polynomial time if $\langle S_{\sigma_i}, H_G\rangle^{\mathrm{O}^{\sigma_i}(G)}=\langle S_{\sigma_i}, H_G\rangle$.

\begin{algorithm}[H]
\caption{IsSigmaPermutable($G, H, K, \sigma$)}
\SetAlgoLined
\KwResult{True, if $H/K$ is $\sigma$-$p$-permutable in $G/K$ and false otherwise.}
\KwData{$H=\langle S\rangle$ is a  subgroup of a group $G$, $K\trianglelefteq G$, $\sigma$ is a partition of $\pi(G)$}

Compute $H_G$\;

\If{not IsSigmaNilpotent$(H^G, H_G, \sigma)$}{\Return{False}\;}

Compute $S_{\sigma_1},\dots, S_{\sigma_k}$\;

\For{$\sigma_i\in\sigma(H/H_G)$}
  {$T\gets \langle S_{\sigma_i}, H_G\rangle$ (i.e the full inverse image of $\mathrm{O}_{\sigma_i}(H/H_G)$)\;
   \If{$T^{\mathrm{O}^{\sigma_i}(G)}\neq T$}{\Return{False}\;}
}

 \Return{True\;}
 \end{algorithm}

4.  A composition series $G_0=K\triangleleft G_1\triangleleft\dots\triangleleft G_k=G$ of $G$ passing through $K$ can be found in  polynomial time by 2 of Theorem \ref{Basic}. Now it is enough to check that $\pi(|G_i|/|G_{i-1}|)\subseteq \sigma_j$ for some $\sigma_j\in\sigma$ for all $i\in\{1,\dots,k\}$.

Note that if  $G/K$ is $\sigma$-soluble, then  it has a Hall $\sigma_i$-subgroup for every $\sigma_i\in\sigma$ by \cite[Theorem B]{Skiba2015}. Now sets of   $\sigma$-permutable subgroups and $\sigma$-$p$-permutable subgroups of $G/K$ coincide by Lemma \ref{coicide}. Therefore we can use 3 as the check for $\sigma$-permutability.

\section{Proof of Theorem \ref{thm3}}

1. The least $\sigma$ for which a group $G$ is $\sigma$-nilpotent can be found in  polynomial time.

Let $\sigma$ be the least partition of $\pi(G/K)$ for which $G/K$ is $\sigma$-nilpotent.
Let $S$ be a generating set of a group $G$ and $\pi(G/K)=\{p_1,\dots, p_n\}$. We start with the least possible partition $\sigma^0=\{\{p_1\},\dots,\{p_k\}\}$ of $\pi(G/K)$. From $G\leq S_n$ it follows that $k\leq n$. It is clear that $\sigma^0\leq\sigma$. Let $\pi=\pi(G)\setminus\pi(G/K)$. According to Lemma \ref{decomp} given a generating set $S$ (of polynomial in $n$ size) of $G$ one can find in polynomial time the generating set $S'=\cup_{\{p_i\}\in\sigma}S_{\{p_i\}}\cup S_{\pi}$  (of polynomial in $n$ size) such that every its element is a $p_i$-element for some $\{p_i\}\in\sigma$ or $\pi(G/K)'$-element.
  Let $S_{\sigma^0_i}=\{s_{p_i}\mid s\in S\}$.

Assume that we know some partition $\sigma^1\leq \sigma$ and corresponding to it sets of generators $S_{\sigma^1_i}$. Our idea is to construct partition $\sigma^2$ such that $\sigma^1\leq\sigma^2\leq\sigma$ and if $\sigma^1=\sigma^2$, then $\sigma^2=\sigma$. Note that every chain of partitions from $\sigma^0$ to $\sigma$ has at most $n$ elements.

Let $\Gamma=(V, E)$ be a graph where  $V=\sigma^1$ and two vertices are connected by the edge iff either $[S_{\sigma^1_i}, S_{\sigma^1_j}]\not\subseteq K$ or
$\pi(|\langle S_{\sigma^1_i}\rangle|/|\langle S_{\sigma^1_i}\rangle\cap K|)\cap \sigma^1_j\neq\emptyset$ or $\pi(|\langle S_{\sigma^1_j}\rangle|/|\langle S_{\sigma^1_j}\rangle\cap K|)\cap \sigma^1_i\neq\emptyset$. From the definition of $\sigma$-nilpotency and $\sigma^1\leq \sigma$ it follows that two elements of $\sigma^1$ lie in the same element of $\sigma$ if they are connected by an edge in $\Gamma$. Therefore
two elements of $\sigma^1$ lie in the same element of $\sigma$ if they are in the same connected component of $\Gamma$.
By finding connected components of $\Gamma$ and joining all vertices in the same connected component we will find  a partition $\sigma^2$ with  $\sigma^1\leq\sigma^2\leq\sigma$.
Note that there are no more than $n(n-1)/2$  pairs of vertices in $\Gamma$ and for a given pair of vertices we can check if it is joined by an edge in  polynomial time by Theorem \ref{Basic}. Also finding connected components of $\Gamma$ can be done by the breadth first search in  polynomial time.

If there are no edges in $\Gamma$, then $\sigma^1=\sigma^2$, $\pi(|\langle S_{\sigma^1_i}\rangle|/|\langle S_{\sigma^1_i}\rangle\cap K|)\subseteq \sigma^1_i$ for all $\sigma^1_i\in\sigma^1$ and $[S_{\sigma^1_i}, S_{\sigma^1_j}]\subseteq K$ for all $\sigma^1_i, \sigma^1_j\in\sigma^1$ with $\sigma_i^1\neq\sigma_j^1$. It means that $G/K$ is $\sigma^1$-nilpotent by Theorem \ref{thm4}. From $\sigma^1\leq \sigma$ it follows that $\sigma^1=\sigma$.
If    $\Gamma$ has edges, then we can set $\sigma^1=\sigma^2$ and repeat the previous step. In no more than $n$ steps we will stop.

\begin{algorithm}[H]
\caption{LeastSigmaNilpotent($G, K$)}
\SetAlgoLined
\KwResult{The least partition $\sigma$ of $\pi(G/K)$ for which $G/K$ is $\sigma$-nilpotent.}
\KwData{$K$ is a  subgroup of a group $G$}

\For{$p\in\pi(|G|/|K|)$}
  {Add $\{p\}$ to $\sigma$\;}

$\pi\gets\pi(G)\setminus\pi(G/K)$

Using partition $\{\sigma_1,\dots,\sigma_k, \pi\}$ of $\pi(G)$ compute $S_{\sigma_1},\dots, S_{\sigma_k}$\;

 $\sigma^1\gets\{\pi(G/K)\}$\;

\While{$\sigma\neq\sigma^1$}
{$\sigma^1\gets \sigma$\;

Define a graph $\Gamma=(V, E)$ with empty $V$ and $E$\;

\For{$\sigma_i\in\sigma$}
     {Add $\sigma_i$ to $V$ with label $l_i=\pi(|\langle S_{\sigma_i}\rangle|/|\langle S_{\sigma_i}\rangle\cap K|)$\;}

\For{$i\in\{1,\dots, k-1\}$}
  {\For{$j\in\{i+1,\dots k\}$}
       {\If{$[S_{\sigma_i}, S_{\sigma_j}]\not\subseteq K$ \textbf{or} $l_i\cap \sigma_j\neq\emptyset$ \textbf{or} $l_j\cap \sigma_i\neq\emptyset$}
             {Add $\{\sigma_i, \sigma_j\}$ to $E$\;}
       }
  }

   Find connected components $\Gamma_k$, $1\leq k\leq l$, of graph $\Gamma$\;

\If{$|\sigma|=l$}{\Return{$\sigma$}\;}

   $|\sigma^2|\gets l$\;
   \For{$k\in\{1,\dots, l\}$}
      {$\sigma^2_k\gets\{\}$\;
         $T_k\gets []$\;
      \For{$\sigma_i\in \Gamma_k$}
          {$\sigma^2_k\gets\sigma^2_k\cup l_i$\;
           $T_k\gets  T_k \cup S_{\sigma_i}$\;
   }
   $S_{\sigma_k}\gets T_k$\;
   }
   $\sigma\gets\sigma^2$\;
 }
 \Return{$\sigma$\;}
 \end{algorithm}

2. Let $\sigma$ be the least partition of $\pi(G/K)$ for which $G/K$ is $\sigma$-soluble.
Find a composition series $G_0=K\triangleleft G_1\triangleleft\dots\triangleleft G_k=G$ of $G$ passing through $K$ (it can be done in  polynomial time by 2 of Theorem \ref{Basic}). It is clear that if $\pi(G_i/G_{i-1})\cap\pi(G_j/G_{j-1})\neq\emptyset$, then there is $\sigma_k\in\sigma$ with $\pi(G_i/G_{i-1})\cup\pi(G_j/G_{j-1})\subseteq\sigma_k$.

Let  $\sigma^0_i=\pi(|G_i|/|G_{i-1}|)$ and $\Gamma=(V, E)$ be a graph where $V=\{\sigma^0_1,\dots, \sigma^0_k\}$ and $\sigma^0_i$ is joined by an edge with $\sigma^0_j$ iff $\sigma^0_i\cap \sigma^0_j\neq\emptyset$. Note that  $|V|\leq 2n$ by Lemma \ref{chain}. Hence $\Gamma$ can be computed in  polynomial time. The connected components  $\Gamma_i$, $1\leq i\leq k$ of  $\Gamma$ can be computed in   polynomial time by the breadth first search algorithm. Let $\sigma^1_i=\cup_{\sigma^0_j\in\Gamma_i}\sigma^0_j$. It is clear that $\sigma^1=\{\sigma^1_1,\dots, \sigma^1_k\}=\sigma$ is the required partition.

3. Assume that $\sigma$ is the least partition of $\pi(G)$ for which $H$ is $\sigma$-$p$-permutable in $G$, $K\trianglelefteq G$, $K\leq H$ and $\sigma^0$ is the least partition of $\pi(G/K)$ for which $H$ is $\sigma$-$p$-permutable in $G/K$. Let prove that  $\sigma$ can be obtained from $\sigma^0$ by adding $\{p\}$ to it for all $p\in\pi(G)\setminus\pi(G/K)$.

Let $\sigma^1=\{\sigma_i\mid \sigma_i\in\sigma^0$ or $\sigma_i=\{p\}$ for $p\in\pi(G)\setminus\pi(G/K)\}$. Since $H/K$ is $\sigma^0$-$p$-permutable in $G/K$,  $H$ permutes with every $\sigma_i$-projector of $G$ for all $\sigma_i\in\sigma^0$ by Lemma~\ref{l1}. Note that $K\leq H$ contains every Sylow $p$-subgroup ($\{p\}$-projector) of $G$ for all $p\in\pi(G)\setminus\pi(G/K)$. Hence $H$ permutes with them. Thus $H$ is $\sigma^1$-$p$-permutable in $G$. Now $\sigma\leq\sigma^1$. Assume that $\sigma\neq\sigma^1$, i.e. some elements of $\sigma^1$ are disjoint unions of elements of $\sigma$. From the construction of $\sigma^1$ it follows that  some elements of $\sigma^0$ are disjoint unions of elements of $\sigma$. From Lemma \ref{l1} it follows that $H/K$ permutes with every $\sigma_i$-projector of $G/K$ for all $\sigma_i\in\sigma$. This contradicts the fact that $\sigma^0$ is the least partition of $\pi(G/K)$ for which $H$ is $\sigma^0$-$p$-permutable. Thus $\sigma^1=\sigma$.

Note that $(H/K)_{G/K}=H_G/K$ and $G/H_G\simeq (G/K)/((H/K)_{G/K})$. Hence to compute   the least partition $\sigma$ of $\pi(G/K)$ for which $H/K$ is $\sigma$-$p$-permutable in $G/K$ we can compute   the least partition $\sigma^0$ of $\pi(G/H_G)$ for which $H/H_G$ is $\sigma^0$-$p$-permutable in $G/H_G$ and then add $\{p\}$ to it for all $p\in\pi(G/K)\setminus\pi(G/H_G)$.

Let $H$ be a subgroup of $G$,  $\sigma$ be the least partition of $\pi(G/H_G)$ for which $H/H_G$ is $\sigma$-$p$-permutable in $G/H_G$.

According to Lemma \ref{t1}  $H^G/H_G$ is $\sigma$-nilpotent. In particular, if $\sigma^{-1}$ is
the least partition of $\pi(H^G/H_G)$ for which $H^G/H_G$ is $\sigma^{-1}$-nilpotent, then by adding $\{p\}$ for all $p\in\pi(G/H_G)\setminus\pi(H^G/H_G)$ to $\sigma^{-1}$ we obtain the least partition $\sigma^0$ of $\pi(G/H_G)$ for which $H^G/H_G$ is $\sigma^0$-nilpotent. Note that $\sigma^{0}\leq\sigma$ and $\sigma^0$ can be computed in  polynomial time by 1.

Now let $\sigma^{1}$ be some partition with $\sigma^{0}\leq \sigma^{1}\leq\sigma$ (we can chose $\sigma^{1}=\sigma^{0}$).
Therefore every element of $\sigma$ is the join of some elements of $\sigma^{1}$.

Since $H/H_G$ is $\sigma$-nilpotent and $\sigma^{1}$-nilpotent for $\sigma^{1}\leq\sigma$, we see that every Hall $\sigma_i$-subgroup $H_i/H_G$ is the direct product of some Hall $\sigma_{i,j}^{1}$-subgroups of $H_{i, j}/H_G$ of $H/H_G$ where $\sigma_i=\cup_j\sigma_{i,j}^{1}$. Note that if some subgroup $T/H_G$ normalizes $H_i/H_G$, then it normalizes every subgroup $H_{i,j}/H_G$.

According to Lemma \ref{l2}  $\mathrm{O}^{\sigma_i}(G/H_G)\leq N_{G/H_G}(H_i/H_G)$.
Note that \begin{center}
  $\mathrm{O}^{\pi_1\cap\pi_2}(G)=\mathrm{O}^{\pi_1}(G)\mathrm{O}^{\pi_2}(G)$ for any $\pi_1, \pi_2\subseteq\mathbb{P}$.
\end{center} Now $\mathrm{O}^{\sigma_i}(G/H_G)=\prod_{k\neq i}\mathrm{O}^{\sigma_k'}(G/H_G)=\prod_{k\neq i}\prod_j\mathrm{O}^{{\sigma^{1}_{k,j}}'}(G/H_G)$. Thus if $a, b\in\sigma^{1}$ and $\mathrm{O}^{b'}(G/H_G)\not\leq N_G(H_a/H_G)$, then $a, b$ belongs to the same element of $\sigma$.

Now consider a graph $\Gamma$ whose vertices are elements of $\sigma^{1}$ and vertices $a, b$ are connected by an edge if
$\mathrm{O}^{b'}(G/H_G)\not\leq N_G(H_a/H_G)$ or $\mathrm{O}^{a'}(G/H_G)\not\leq N_G(H_b/H_G)$. All vertices of the same connected component of $\Gamma$ must belong to some element of $\sigma$. Hence, by joining the sets which correspond to the vertices of the same connected component of $\Gamma$ we obtain new partition $\sigma^{2}$ of $\pi(G)$ with $\sigma^{0}\leq\sigma^{2}\leq \sigma$. Note that $|V(\Gamma)|\leq n$. From $\mathrm{O}^{b'}(G/H_G)=\mathrm{O}^{b'}(G)H_G/H_G$ it follows that $\mathrm{O}^{b'}(G/H_G)\not\leq N_G(H_a/H_G)$ iff $H_a^{\mathrm{O}^{b'}(G)}\neq H_a$. The last condition can be checked in  polynomial time by Theorem \ref{Basic} and Remark \ref{rem1}. Hence the connected components of $\Gamma$ can be computed in  polynomial time. If $\sigma^{2}\neq\sigma^{1}$. then we can set $\sigma^{1}\leftarrow\sigma^{2}$ and repeat the previous step. Note that after no more that $|\sigma^{0}|\leq n$ steps we obtain  $\sigma^{2}=\sigma^{1}$.

Let prove that $\sigma^{1}=\sigma$. We need only to prove that $H$ is  $\sigma^{1}$-$p$-permutable subgroup of $G$. Let $H_i$ be a Hall $\sigma_i^{1}$-subgroup of $H/H_G$. Then $\mathrm{O}^{{\sigma_i^1}}(G/H_G)\leq N_{G/H_G}(H_i/H_G)$ by the construction of $\sigma^{1}$.
It means that $H_i/H_G$ permutes with every $\sigma_i'$-subgroup of $G/H_G$.
Since $H^G/H_G$ is a normal $\sigma^{1}$-nilpotent subgroup of $G/H_G$, $H_i/H_G\leq \mathrm{O}_{\sigma^1_i}(H^G/H_G)\leq\mathrm{O}_{\sigma^1_i}(G/H_G)$. Hence $H_i/H_G$ belongs to every $\mathfrak{G}_{\sigma_i}$-projector of $G/H_G$ and thus permutes with it.
It means that  $H_i/H_G$ is a $\sigma^{1}$-$p$-permutable subgroup of $G/H_G$.

\begin{algorithm}[H]
\caption{LeastSigmaPermutable($G, H$)}
\SetAlgoLined
\KwResult{The least partition $\sigma$ of $\pi(G/K)$ for which $H/K$ is $\sigma$-$p$-permutable in $G/K$.}
\KwData{$H$ is a  subgroup of a group $G$, $K\trianglelefteq G$ with $K\leq H$}

$\sigma\gets $LeastSigmaNilpotent$(H^G, H_G)$\;
\If{$\sigma=\{\pi(G/K)\}$}{\Return{$\sigma$}\;}

\For{$p\in\pi(G/H_G)\setminus\pi(\sigma)$}
  {Add $\{p\}$ to $\sigma$\;}

\For{$\sigma_i\in\sigma$}
   {$P_i\gets\mathrm{O}^{\sigma_i'}(G)$\;}

$\sigma^1\gets\{\pi(G)\}$\;

\While{$\sigma\neq\sigma^1$}
{$\sigma^1\gets \sigma$\;

   \For{$\sigma_i \in \sigma^1$}
      {$H_i\gets $ the full inverse image of $\mathrm{O}_{\sigma_i}(H/H_G)$ in $H/H_G$\;}

    \For{$\sigma_i \in \sigma$}
       {\For{$\sigma_j \in \sigma, i\neq j$}
          {\If{$H_i^{P_j}\neq H_i$}
               {$\sigma_i\gets\sigma_i\cup\sigma_j$\;}
          }
        }
   Find connected components $\Gamma_k$, $1\leq k\leq l$, of graph $\Gamma$ with $V(\Gamma)=\sigma$ and $\{\sigma_i, \sigma_j\}\in E(\Gamma)$   iff $i\neq j$ and $\sigma_i\cap\sigma_j\neq\emptyset$\;
\If{$|\sigma|=l$}{\Return{$\sigma$}\;}

   $|\sigma^2|\gets l$\;
   \For{$k\in\{1,\dots, l\}$}
      {$\sigma^2_k\gets\{\}$\;
         $T_k\gets []$\;
         \For{$\sigma_i\in \Gamma_k$}
               {$\sigma^2_k\gets\sigma^2_k\cup \sigma_i$\;
                 $T_k\gets\langle T_k, P_i\rangle$\;
                }
  $P_k\gets T_k$\;}

   $\sigma\gets\sigma^2$\;

}
\For{$p\in\pi(G/K)\setminus\pi(G/H_G)$}{Add $\{p\}$ to $\sigma$\;}
 \Return{$\sigma$\;}
 \end{algorithm}

{\small\bibliographystyle{siam}
\bibliography{SigmaAlg}}

\end{document}